\documentclass[11pt,twoside]{amsart}
\usepackage{color}
\usepackage{epic}
\usepackage[latin1]{inputenc}
\usepackage[OT1]{fontenc}
\usepackage{amscd}
\usepackage{amsmath}
\usepackage{amsthm}
\usepackage{amssymb}
\usepackage[all]{xy}
\usepackage{xypic}
\xyoption{curve}
\usepackage{ifthen} 
\usepackage{hyperref} 
\usepackage{graphicx}
\usepackage{multirow}

\newtheorem{theorem}{Theorem}[section]

\newtheorem{proposition}[theorem]{Proposition}

\newtheorem{lemma}[theorem]{Lemma}

\newtheorem{question}[theorem]{Question}

\numberwithin{equation}{section}

\theoremstyle{definition}
\newtheorem{definition}[theorem]{Definition}
\newtheorem{remark}[theorem]{Remark}
\newtheorem{example}[theorem]{Example}

\newcommand{\IA}{\mathbb{A}}
\newcommand{\CC}{\mathbb{C}}

\newcommand{\FF}{\mathbb{F}}

\newcommand{\PP}{\mathbb{P}}
\newcommand{\QQ}{\mathbb{Q}}

\newcommand{\ZZ}{\mathbb{Z}}

\renewcommand{\to}{\xymatrix@1@=15pt{\ar[r]&}}
\renewcommand{\rightarrow}{\xymatrix@1@=15pt{\ar[r]&}}
\renewcommand{\mapsto}{\xymatrix@1@=15pt{\ar@{|->}[r]&}}
\renewcommand{\twoheadrightarrow}{\xymatrix@1@=15pt{\ar@{->>}[r]&}}
\renewcommand{\hookrightarrow}{\xymatrix@1@=15pt{\ar@{^(->}[r]&}}
\newcommand{\congpf}{\xymatrix@1@=15pt{\ar[r]^-\sim&}}

\xyoption{arrow}
\xyoption{curve}
\newdir{ (}{{}*!/-5pt/@^{(}}

\begin{document}

\newboolean{xlabels} 
\newcommand{\xlabel}[1]{ 
                        \label{#1} 
                        \ifthenelse{\boolean{xlabels}} 
                                   {\marginpar[\hfill{\tiny #1}]{{\tiny #1}}} 
                                   {} 
                       } 
\setboolean{xlabels}{false} 

\title[On uniformly rational varieties]{On uniformly rational varieties}

\author[Bogomolov]{Fedor Bogomolov$^1$}
\address{Feder Bogomolov, Courant Institute of Mathematical Sciences\\
251 Mercer Street\\
New York, NY 10012, USA}
\email{bogomolo@cims.nyu.edu}
\author[B\"ohning]{Christian B\"ohning$^2$}
\address{Christian B\"ohning, Fachbereich Mathematik der Universit\"at Hamburg\\
Bundesstra\ss e 55\\
20146 Hamburg, Germany}
\email{christian.boehning@math.uni-hamburg.de}

\thanks{$^1$ Supported by NSF grant DMS- 1001662 and by AG Laboratory GU- HSE grant RF government ag. 11 11.G34.31.0023}
\thanks{$^2$ Supported by Heisenberg-Stipendium BO 3699/1-1 of the DFG (German Research Foundation)}

\begin{abstract}
We investigate basic properties of \emph{uniformly rational} varieties, i.e. those smooth varieties for which every point has a Zariski open neighborhood isomorphic to an open subset of $\mathbb{A}^n$. It is an open question of Gromov whether all smooth rational varieties are uniformly rational. We discuss some potential criteria that might allow one to show that they form a proper subclass in the class of all smooth rational varieties. Finally we prove that small algebraic resolutions and big resolutions of nodal cubic threefolds are uniformly rational.  
\end{abstract}

\maketitle

\section{Introduction}\xlabel{sIntroduction}

In \cite{Gro89}, p.885,  Misha Gromov raised the following question:

\begin{question}\xlabel{qGromov}
Let $X$ be a smooth (complete) rational variety, of dimension $n$. Is it then true that for \emph{every} point $x\in X$ one can find a Zariski open neighborhood $U\ni x$ which is biregular to an open subset of $\IA^n$? 
\end{question}

We will say that $X$ is \emph{uniformly rational} if Gromov's question has a positive answer for $X$. In other words, the question is thus if an open subset of $\IA^n$ could possibly have some strange smooth compactification which, locally around certain points on the boundary components, looks totally different from opens in affine space.

Recall from \cite{Nak70} and \cite{FuNa71} that there is a simple numerical criterion that a submanifold  $E$ in a complex manifold  $X$ can be blown-down (in the analytic category), i.e. occurs as the exceptional divisor of a blow-up of a manifold $X'$ in a submanifold of codimension at least $2$: if $E$ has a structure of an analytic fibre bundle $\pi : E \to M$ with fibres $\pi^{-1}(m) \simeq \PP^r$, some $r \ge 1$, then this is possible if and only if $\mathcal{O}_{X}(E)|_{\pi^{-1}(m)} \simeq \mathcal{O}_{\PP^r}(-1)$ for every $m\in M$. Moreover, the smallest class of compact complex manifolds containing complete smooth algebraic varieties and stable under such analytic blow-downs is the class of Moishezon manifolds, see \cite{Moi67}. Below we will see that the property of being uniformly rational is stable under blowing up, and that the main difficulty is to prove a converse: if $X \to X'$ is a blow-down and $X$ is uniformly rational, decide if $X'$ is also uniformly rational. Hence it may be interesting to generalize Question \ref{qGromov} to

\begin{question}\xlabel{qGromov2}
Let $X$ be a compact Moishezon manifold which is bimeromorphic to $\PP^n$. Is it then true that for \emph{every} point $x\in X$ one can find a bimeromorphic map $X \dasharrow \PP^n$ defined at $x$?
\end{question}

Here is the road map of this paper. In Section \ref{sGeneralTheory}, after discussing several classes of rational varieties which are uniformly rational and the proof of the fact (Proposition \ref{pBlowup}) that the blow-up of a uniformly rational variety along a smooth center is again uniformly rational, we show in Proposition \ref{pRectifyNecessary} that the coincidence of the classes of rational and uniformly rational varieties would have strong implications on the rectifiability by Cremona transformations of certain families of rational subvarieties of projective space. We then discuss the potential usefulness of this criterion in the light of recent results of Mella and Polastri on the Cremona group.

In Section \ref{sNodalCubics} we investigate the uniform rationality of a nontrivial family of threefold examples: small and big resolutions of nodal cubic threefolds. We prove that Question \ref{qGromov} always has an affirmative answer for these, but Question \ref{qGromov2} has a negative answer for some small resolutions that are Moishezon, but not algebraic.


\section{General theory}\xlabel{sGeneralTheory}

As noted already in \cite{Gro89}, in many concrete examples, it is straightforward to check that the variety at hand is uniformly rational.

\begin{example}\xlabel{eHomogenous}
If $X$ is rational-homogeneous, then clearly $X$ is uniformly rational. Algebraic vector bundles over a uniformly rational $X$ are uniformly rational. 
Also it is clear that all smooth rational surfaces are uniformly rational: this follows from the fact that every such surface is the blow-up of $\PP^2$, $\PP^1\times \PP^1$ or the Hirzebruch surface $\FF_r$ with $r\ge 2$. 
\end{example}

\begin{example}\xlabel{eToric}
It is clear that smooth toric varieties are uniformly rational because every cone in their fans can be generated by a subset of a basis for the free abelian groups $N$. So every toric chart is an open subset of $\IA^n$ there.
\end{example}

\begin{definition}\xlabel{dQuasiHomogeneous}
We say that a variety $X$ is \emph{quasi-homogeneous} if for any Zariski dense open subset $U \subset X$ and every point $x\in X$ there exists a birational self map $f : X \dasharrow X$ which is defined at $x$ and an isomorphism onto its image in a neighborhood of $x$ mapping the point $x$ into $U$. 
\end{definition}

If a rational variety is quasi-homogeneous, it is clearly uniformly rational. Let us discuss two examples in some more detail. They appear already in \cite{Gro89}, though without details and in slightly less generality. 

\begin{example}\xlabel{eCubics}
Smooth rational cubic hypersurfaces $X$ in $\PP^N$ are quasi-homogeneous, hence uniformly rational (conjecturally, $N$ has to be even for rational $X$'s with $N\ge 3$ to exist at all, and $X$ has to be rather special, e.g. containing two disjoint linear spaces of half the dimension of the cubic). Namely, for given $U \subset X$ Zariski open dense and given $x\in X$, one may consider the lines through $x$ and points of $U$. A generic such line will intersect $X$ in a third point $p$. Now reflection in that point gives the desired birational self map carrying a neighborhood of $x$ isomorphically into $U$. Here reflection means that a generic point $p_1$ is mapped to the third intersection point with $X$ of the line through $p_1$ and $p$.  

In fact, this argument can be generalized to show that the smooth locus of a singular cubic hypersurface $X$ in $\PP^N$ is uniformly rational. For suppose $x\in X$ is smooth and $U \subset X$ Zariski open dense, consisting of smooth points. Then a general line through $x$ and a point in $U$ will intersect $X$ in a third point $p$: for this it is sufficient that not all tangent hyperplanes to points in $U$ pass through $x$, and this is not the case since otherwise $X$ would be a cone over $x$. 
\end{example}

\begin{example}\xlabel{eQuadrics}
Another example of a quasi-homogeneous variety is the smooth complete intersection $X$ of two quadrics in $\PP^{n+2}$: pick two general points $p_1$, $p_2$ on $X$, consider planes $E_{p_1, p_2}$ through them, and look at the birational map interchanging the two points of intersection of $E_{p_1, p_2}$ with $X$ away from $p_1, p_2$. If $p_1$, $p_2$ are chosen suitably general this will map a neighborhood of a  given point $x\in X$ isomorphically into a given dense open subset $U\subset X$. The same holds if $X$ has  singularities: then the smooth locus $X^0$ is quasi-homogeneous. Namely, given $x\in X$ smooth and a Zariski dense open set $U \subset X$, a general plane $E$ through $x$ and a smooth point $x'$ of $U$ will intersect $X$ transversally at $x$ and $x'$ (for this it suffices that the tangent space $T_{x'}X$ does not contain $x$, which is true for a general $x'\in U$ since otherwise $x$ would be singular). Moreover, a general $E$ through $x$ and $x'$ will intersect $X$ in two additional points $p_1$ and $p_2$ away from $x$ and $x'$. This is so because projection away from the line $l$ joining $x$ to $x'$ exhibits $X$ as a two to one cover of $\PP^{n}$ which cannot be everywhere ramified. 
\end{example}

Let us now recall the proof of the fact that uniform rationality is stable under blowing up nonsingular centers (\cite{Gro89}, Proposition p. 885) since the geometric ideas in the argument will be frequently used below. 

\begin{proposition}\xlabel{pBlowup}
Let $X$ be uniformly rational and let $Y \subset X$ be a smooth subvariety, Then $\mathrm{Bl}_Y (X)$ is uniformly rational.
\end{proposition}

\begin{proof}
Let $y\in Y$, then by hypothesis, $y$ has a neighborhood isomorphic to an open subset of affine space, so one may also suppose from the beginning that $Y \subset \IA^n$ of dimension $m \le n-2$. There is a birational map $\varphi : \IA^n \dasharrow \IA^n$ defined at $y$ and mapping $Y$ birationally to a hypersurface in $L\simeq \IA^{m+1} \subset \IA^n$: indeed, choose a generic splitting $\IA^n = L \oplus M$, $L\simeq \IA^{m+1}$, $M \simeq \IA^{n-m-1}$ and consider the projection $\pi_M : \IA^n \to L$. Then $\pi_M$ restricted to $Y$ will be an isomorphism of a neighborhood of $y$ onto $\pi_M (Y) \cap U$ where $U\subset L$ is some affine open subset. In other words, the trivial $\IA^{n-m-1}$ bundle $U\oplus \IA^{n-m-1}$ has a section $Y$ over $\pi_M( Y) \cap U$ and this extends (since $U$ is affine) to a section $\sigma : U \to U \oplus \IA^{n-m-1}$. The map $\varphi$ can then be given by 
\[
\varphi (u , m) = (u, m - \sigma (u) ) , \quad u \in U, m \in M .
\]
Thus we assume now that $L\simeq \IA^{m+1}$ is the coordinate subspace with coordinates $x_1, \dots , x_{m+1}$ and $M \simeq \IA^{n-m-1}$ the one with coordinates $x_{m+2}, \dots , x_{n}$, and $Y \subset L$ defined by $f(x_1, \dots , x_{m+1}) =0$. The blow-up $\mathrm{Bl}_Y \IA^n$ is given as the closure of the graph of the map $\IA^n \dasharrow \PP^k$, $k =n-m-1$, sending $(x_1, \dots , x_n)$ to 
\[
(Y_0, \dots , Y_k) = (f, x_{m+2}, \dots , x_{n})
\]
where $Y_0, \dots , Y_k$ are homogeneous coordinates on $\PP^k$.  The open subset $V$ of $\mathrm{Bl}_Y (\IA^n)$ given by $Y_0 \neq 0$ can then be described as
\[
V = \left\{  ((x_1, \dots , x_n), (y_1, \dots , y_k)) \subset \IA^n \times \IA^k \mid  x_{m+2} = y_1 f, \dots , x_n = y_k f \right\}
\]
where $y_i = Y_i/Y_0$ are affine coordinates. Projection of $V$ to the coordinates $(x_1, \dots , x_{m+1}, y_1, \dots , y_k)$ gives an isomorphism $V \simeq \IA^n$. This gives us the required coordinate neighborhoods for all points of $\mathrm{Bl}_Y (\IA^n)$ not contained in $Y_0=0$. Notice however, that one can choose, for given point $p \in \{Y_0=0\} \subset \mathrm{Bl}_Y (\IA^n)$, an automorphism $\psi$ of $\IA^n$ mapping $Y \to Y'$ and by functoriality $\mathrm{Bl}_Y (\IA^n)$ to $\mathrm{Bl}_{Y'} (\IA^n)$ such that if we perform the whole  construction in the coordinates $x_1', \dots , x_n'$, $Y_0', \dots , Y_k'$ for the image, then $\psi (p)$ will not be contained in $\{ Y_0' =0 \}$. We just have to choose $\psi$ in such a way (as some rotation for example) that it moves a given direction away from the set of directions corresponding to $Y_0 =0$ 
This concludes the proof.
\end{proof}

Since by the Weak Factorization Theorem, one can factor every birational map between smooth projective varieties in a series of blow-ups and blow-downs along nonsingular centers, it is clear that one has to control how the property of being uniformly rational behaves under blow-down. 

\begin{remark}\xlabel{rCompleteIntersections}
It is interesting to notice that the blow-up of a regular scheme in a complete intersection scheme is nonsingular if and only if the complete intersection scheme was already regular, see \cite{CaVa97}, Thm. 2.1. Hence, if we start blowing up uniformly rational $X$ in more general subschemes $Z \subset X$, but of course under the hypothesis that $X' = \mathrm{Bl}_Z (X)$ be again smooth, we have to allow for something more general than complete intersections for $Z$ to go beyond Proposition \ref{pBlowup}. An interesting example is the blow up of $\IA^4$ in the cone $C$ over a twisted cubic in $\PP^3$; amusingly, this is smooth \cite{CaVa97}, Section 4. 
\end{remark}

The proof of Proposition \ref{pBlowup} allows us to formulate certain geometric consequences that the equality of the classes of smooth rational and smooth uniformly rational varieties would entail.  We first need some terminology.

\begin{definition}\xlabel{dRectify}
\begin{itemize}
\item[(1)]
A \emph{divisorial family of rational varieties of dimension} $r$ in a variety $Z$ is a map $\varphi : \PP^r \times B \dasharrow Z$
which is birational unto its image and where $\dim B = \dim Z - r -1$. 
\item[(2)]
A divisorial family of rational varieties of dimension $r$ as in (1) with $Z = \PP^n$ is said to be \emph{rectifiable} if there exists a birational automorphism $\gamma : \PP^n \dasharrow \PP^n$ (also called a Cremona transformation) such that the birational transform under $\gamma$ of $\overline { \varphi (\PP^r \times \{ b \}) }$ is a line in $\PP^n$ for a general $b\in B$.
\item[(3)]
A divisorial family of rational varieties of dimension $r$ as in (1) is said to be \emph{contractible} if there exists a smooth model $Z'$ of $Z$ and a blow down $\sigma :  Z\simeq \mathrm{Bl}_Y (X) \to X$, $X$ smooth, with exceptional divisor $E\to Y$ such that the birational transform of $\overline{ \varphi (\PP^r \times B ) }$ on $Z'$ is $E$ and the birational transform of $\overline { \varphi (\PP^r \times \{ b \}) }$, for general $b$, is a fibre of $\sigma |_E : E \to \sigma (E)=Y \subset X$. 
\end{itemize} 
\end{definition}

\begin{proposition}\xlabel{pRectifyNecessary}
If the class of rational algebraic manifolds coincides with the class of uniformly rational algebraic manifolds, then every \emph{contractible} divisorial family $\varphi : \PP^r \times Y \dasharrow \PP^n$ of rational varieties of dimension $r$ in $\PP^n$ is \emph{rectifiable}.
\end{proposition}

\begin{proof}
Let $\sigma : \mathrm{Bl}_Y (X) \to X$ be as in Definition \ref{dRectify} for the given family $\varphi$. By assumption, $X$ is uniformly rational, so we get an open subset $V \subset \mathrm{Bl}_Y (X)$ and coordinates $(x_1, \dots , x_{m+1}, y_1, \dots , y_k)$ on $V$ as in the proof of Proposition \ref{pBlowup}. These coordinates, precomposed with the map $Z \dasharrow Z' = \mathrm{Bl}_Y (X)$, give the required birational map $\gamma : \PP^n \dasharrow \PP^n$ that rectifies the family $\varphi$.
\end{proof}

\begin{remark}\xlabel{rBirationalEmbeddings}
By \cite{MelPol09}, the Cremona group of $\PP^n$ is so flexible that every two birational embeddings of a given variety of codimension at least two into $\PP^n$ are equivalent under it. However, this is no longer so in codimension $1$: there are plane rational sextics with nodes that cannot be transformed into a line (see example at the end of the article by Mella and Polastri). However, these are not contractible in the sense of Definition \ref{dRectify}.
\end{remark}

\begin{remark}\xlabel{rObstructions}
We discuss here the beautiful theory in \cite{MelPol12} and the recent preprint \cite{Mel12}, inspired by the log minimal model program (LMMP) and log Sarkisov theory, with a view toward the possibility of using  Proposition \ref{pRectifyNecessary} to produce examples of not uniformly rational, smooth rational threefolds. 
In \cite{MelPol12} the following are proven:
\begin{itemize}
\item[(1)]
We consider pairs $(X, D)$ with $X$ a normal variety, $D$ a $\QQ$-Weil divisor with $K_X + D$ $\QQ$-Cartier, and we call pairs $(X, D)$ and $(X', D')$ birational if there is a birational map $\varphi : X \dasharrow X'$ and the strict transform $\varphi_{\ast } D$ of $D$ equals $D'$. 

A pair $(\PP^2, C)$, $C$ an irreducible rational curve, is equivalent to $(\PP^2, L)$, $L$ a line, if and only if 
\[
\overline{\kappa } (\PP^2, C) < 0 .
\]
Here $\overline{\kappa} (X, D)$ for a pair $(X, D)$, with $D$ an irreducible reduced divisor, is defined by taking a model $(Y, D_Y)$ of $(X, D)$ with $Y$ and $D_Y$ smooth and putting
\[
\overline{\kappa } (X, D) = \mathrm{tr.deg.}  \bigoplus_m H^0 (Y, m(K_Y +D_Y)) -1 ,
\]
the logarithmic Kodaira dimension of the model $(Y, D_Y)$. 
\item[(2)]
A pair $(\PP^3, S)$, $S$ a rational surface, is equivalent to $(\PP^3, E)$, $E$ a plane, if and only if the sup-effectivity threshold $\overline{\varrho} (\PP^3, S) > 0$ and on some good model $(T, S_T)$ of the pair $(\PP^3, S)$, the divisor $K_T + S_T$ is not pseudoeffective \cite{MelPol12}, Thm. 4.15. Here the terminology is explained as follows: a good model is a pair where $Y$ has terminal $\QQ$-factorial singularities and $D$ is a smooth Cartier divisor. For a good model one defines
\[
\varrho (Y, D) = \mathrm{sup} \{ m \in \QQ | D + m K_Y \mathrm{is} \; \mathrm{effective}\; \QQ \mathrm{-divisor} \} 
\] 
and calls it the effectivity threshold of the pair. The sup-threshold of a pair is then defined as the supremum taken over all good models of the respective effectivity thresholds of the models. 
\item[(3)]
If $(\PP^3, S)$ is equivalent to $(\PP^3, S')$, $S$ and $S'$ surfaces with $d=\mathrm{deg}(S) > d' = \mathrm{deg}(S')$, then the pair $(\PP^3, \frac{4}{d} S)$ is not canonical (i.e. has a maximal center somewhere). See \cite{MelPol12}, Lemma 2.2.
\end{itemize}
Remark that the condition in (1) is not easy to use to prove that a certain pair $(\PP^3, S)$ is \emph{not} equivalent to a $(\PP^3, E)$, $E$ a plane, since it requires control not only over one log resolution of the pair but rather all different good models of the pair in the log category. 

\

In view of Proposition \ref{pRectifyNecessary} we raise the question whether, in dimension $n=2$, some contractible rational curve $C \subset \PP^2$ is rectifiable, and for dimension $3$, whether some contractible divisorial family $\varphi : \PP^1 \times Y \to \PP^3$ with $Y$ a rational curve (to simplify) is rectifiable, hence Cremona equivalent to a rational scroll. Let $S \subset \PP^3$ be the image surface of $\varphi$.
In dimension $2$, we see that the conditions of being contractible and being rectifiable coincide: if $C$ is contractible, or if more generally, $(\PP^n, D)$, is contractible, then $\overline{\kappa} (\PP^n, D) < 0$; namely, there is a model $(Y, D_Y)$ with $Y$ and $D_Y$ smooth, $D_Y$ an exceptional divisor of a blow-up in a smooth center, and $Y$ is rational; hence the result. For curves $C$ this implies rectifiability. Conversely, if $C$ is rectifiable, the pair $(\PP^2, C)$ is equivalent to $(\PP^2, L)$, $L$ a line, and by the invariance of $\overline{\kappa}$ for birational equivalence of pairs, we get $\overline{\kappa} (\PP^2, C) < 0$.

Now if we pass to dimension $n=3$, it is already an interesting open problem to determine if there are any obstructions to a rational surface $S$ being Cremona equivalent to a scroll. A result in \cite{Mel12} shows that every rational \emph{cone} is Cremona equivalent to a plane. This may lead one to speculate at first that maybe any surface $S$ Cremona equivalent to a scroll is Cremona equivalent to a plane. However, as Massimiliano Mella pointed out to us (we are very grateful for this remark), a general linear projection of a smooth scroll into $\PP^3$ will be a surface with ordinary singularities, the highest multiplicity of a singular point being $3$. If $d=\mathrm{deg}(S) > 3$, then (3) above forces a point of multiplicity $> d/2$ which is impossible for $d>5$. 

Some finer invariants, but possibly also related to multiplier ideals in some way, may be necessary to decide if (images of contractible) families $S$ are always equivalent to scrolls for $n=3$.
\end{remark}




\section{Small and big resolutions of nodal cubic threefolds}\xlabel{sNodalCubics}

\subsection{Special cases. Big resolutions}\xlabel{ssSpecialBig}

An interesting class of smooth rational varieties for which uniform rationality does not follow immediately from results above is given by big and small resolutions of nodal cubic threefolds $Y$ in $\PP^4$.  Thus $Y$ is smooth away from a finite number $s$ of nodes which are given in local analytic coordinates by an equation
\[
Q : = \{ x_1x_2 - x_3x_4 = 0 \} . 
\]
Thus $Q$ is a cone in $\IA^4$ over a smooth quadric in $\PP^3$. This has a big resolution $\tilde{Q}=\mathrm{Bl}_{0} (Q)$which replaces the vertex of the cone by $E=\PP^1\times \PP^1$, and two small resolutions $\hat{Q}_1$ and $\hat{Q}_2$, connected by a flop, and obtained from $\tilde{Q}$ by contracting either one of the rulings of $E$. Alternatively, one can describe $\hat{Q}_1$ as $\mathrm{Bl}_{L_1} (Q)$ where $L_1$ is a plane contained in $Q$ spanned by the vertex and a line from a ruling of the smooth quadric in $\PP^3$ over which $Q$ is the cone, and $\hat{Q}_2$ is obtained as $\mathrm{Bl}_{L_2}(Q)$ in the same way but taking a line from the other ruling instead. For example, one may take $L_1 = \{ x_1 = x_3 = 0\}$, $L_2 = \{ x_1 = x_4 = 0 \}$. 

This construction globalizes. See \cite{Fink87} and \cite{FinkWern89} for more information concerning the following facts. The cubic $Y$ has a tangent cone at a chosen node $P$ (the ``watchtower") whose projectivization $Q= \PP (TC_P (Y)) \subset \PP ( T_P (Y)) \simeq \PP^3$ forms a quadric of the (limits of) tangent directions to $Y$ in $P$. The directions corresponding to whole lines through $P$ contained in $Y$ form a curve $C \subset Q \simeq \PP^1 \times \PP^1$ of bidegree $(3,3)$ which is smooth precisely if $Y$ is a \emph{Lefschetz cubic}, i.e. if $P$ is its only singular point. Otherwise the curve $C$ has $s-1$ ordinary double points which correspond bijectively to the nodes of $Y$ away from $P$. More concretely, in suitable homogeneous coordinates $(X_0: \dots : X_4)$ of $\PP^4$ we can write
\[
Y = \{ X_4 F_2 + F_3 = 0 \}
\]
with $F_2$ a polynomial in the variables $X_0$, $X_1$, $X_2$ and $X_3$ only, homogeneous of degree $2$, $F_3$ is homogeneous of degree $3$ and also depends on $X_0$, $X_1$, $X_2$, $X_3$ only,  and $P=(0:0:0:1)$. Hence $F_2=0$ defines the tangent cone in $P$, whose projectivization is by definition a nonsingular quadric in $\PP^3$. The curve $C$ is the complete intersection of the cubic $F_3=0$ with the quadric $F_2=0$ (where we view the variables $X_0, \dots , X_3$ on which $F_2$, $F_3$ depend as homogeneous coordinates in this $\PP^3$). 

Projection away from $P$ defines a morphism
\[
\mathrm{Bl}_P (Y) \to \PP^3
\]
which can be identified with the morphism
\[
\mathrm{Bl}_C (\PP^3 ) \to \PP^3 .
\]
See \cite{Fink87}, Theorem 2.1, for a proof of this statement. 

In particular,

\begin{proposition}\xlabel{pBigRes}
Let $Y$ be a Lefschetz cubic with singular point $P$. Then $\mathrm{Bl}_P (Y)$ is uniformly rational.
\end{proposition}

\begin{proof}
This follows from the isomorphism $\mathrm{Bl}_P (Y) \simeq \mathrm{Bl}_C (\PP^3 )$ and Proposition \ref{pBlowup} because in this case $C$ is smooth. 
\end{proof}

We will treat big resolutions of nodal cubics in general in Subsection \ref{ssBigResolutions}.

Small resolutions of a cubic $Y$ with $s$ nodes are harder to deal with; a priori there are $2^s$ of them, and not all of them need to be projective/algebraic (only Moishezon manifolds). Using some classical geometry, one can immediately say something about the special case of the \emph{Segre cubic} $S_3$ defined by
\[
\xi_0^3 + \xi_1^3 + \xi_2^3 + \xi_3^3 + \xi_4^3 + \xi_5^3 = 0, \quad \xi_1 + \xi_2 + \xi_3 + \xi_4 + \xi_5 =0
\]
where $(\xi_i)$ are homogeneous coordinates in $\PP^5$. One knows classically that the moduli space $\overline{\mathcal{M}}_{0,6}$ of stable $6$-pointed curves of genus $0$ is a small resolution of $S_3$, see \cite{Hunt96}, Chapter 3. 

\begin{proposition}\xlabel{pPointedCurves}
The moduli space $\overline{\mathcal{M}}_{0,n}$ of stable $n$-pointed curves of genus $0$ is uniformly rational for all $n$. In particularly, the small resolution $\overline{\mathcal{M}}_{0,6}$ of the Segre cubic $S_3$ is uniformly rational.
\end{proposition}

\begin{proof}
We use the main result of \cite{Keel92} which gives that $\overline{\mathcal{M}}_{0, n+1}$ can be obtained from 
\[
\overline{\mathcal{M}}_{0, n} \times \overline{\mathcal{M}}_{0, 4} \simeq \overline{\mathcal{M}}_{0, n} \times \PP^1 
\]
as an iterated blow up along smooth codimension 2 subvarieties; then we use Proposition \ref{pBlowup} again and the fact that $\overline{\mathcal{M}}_{0, 5}$ is a del Pezzo surface of degree $5$. 
\end{proof}

\subsection{Algebraic resolutions}\xlabel{ssAlg}

To treat small resolutions of nodal cubic threefolds systematically, we use their classification given in \cite{FinkWern89}, p. 190-198, in terms of the degenerations of the nodal curve $C$ associated to $Y$ (after choosing one node as a distinguished ``watchtower"). 

\begin{remark}\xlabel{rAlgebraicity}
A particular small resolution is in general only a Moishezon manifold. We recall certain results from \cite{FinkWern89}. Let $Y$ be a nodal cubic in $\PP^4$ with $s$ nodes.

\begin{itemize}
\item[(1)]
Depending on which of the two natural directions in the exceptional quadric in a big resolution one blows down locally, there are $2^s$ small resolutions of $Y$ which are a priori only complex manifolds (they are of course Moishezon). Some of these may be isomorphic however. 
\item[(2)]
One has the implications:
\begin{quote}
There exists a projective small resolution\\
$\iff$ all exceptional $\PP^1$'s are not homologous to zero in $H_2 (\hat{Y}, \ZZ)$ for some (and then for every) small resolution $\hat{Y}$ of $Y$\\
$\iff$ every irreducible component of the associated curve is smooth and there is at least one component of bidegree $(a,b)$ with $a\neq b$. 
\end{quote}
The first $\iff$ is loc.cit. Lemma 1.2, the second is Lemma 2.4.
\end{itemize}
\end{remark}

First of all one sees from the classification given in \cite{FinkWern89}

\begin{lemma}\xlabel{lAlgebraic}
There exists a projective small resolution of $Y$ if and only if there exists an algebraic small resolution of $Y$. In each such case, all of the small resolutions are algebraic, however, they are in general only complete algebraic varieties which need not be projective. 
\end{lemma}

\begin{proof}
First of all existence of a projective resolution clearly implies the existence of an algebraic one; if there is no projective resolution, then by Remark \ref{rAlgebraicity}, a node is resolved zero-homologously on every small resolution, hence none of them is algebraic. 

To see that, in the case where a projective resolution exists, all small resolutions are algebraic, one has to use the classification given in \cite{FinkWern89} given in \S 3 in terms of the degenerations of the associated curve. Namely, one has a dichotomy:
\begin{lemma}\xlabel{lFW}
Let $Y$ be a nodal cubic threefold in $\PP^4$ with at least one projective small resolution $X = \hat{Y}$, and let $P\in Y$ be a node. Then $P$ lies on a plane $E$ contained in $Y$ unless we are in case J9 of \cite{FinkWern89}, i.e. $Y$ has no planes and six nodes, and the associated curve $C\subset \PP^1 \times \PP^1$, as seen from any of the nodes, has two smooth irreducible components of bidegrees $(1,2)$ and $(2,1)$ intersecting transversely in $5$ points. 
\end{lemma}
The results follows from the enumeration in of cases in \S 3 of \cite{FinkWern89}. Now we have the following
\begin{lemma}\xlabel{lAlgebraicity}
Let $Y$ be a nodal cubic with precisely six nodes or one where every node lies on a plane. Let $P$ be one of the nodes. Then there exists a Zariski neighborhood $U_P$ of $P$ containing no other nodes and two algebraic varieties $U_{1, P} \to U_P$, $U_{2, P} \to U_P$ where $U_{1, P}$ is isomorphic to the blow-down of one ruling of the exceptional quadric in a big resolution, and $U_{2, P}$ is isomorphic to the blow-down of the other ruling of this exceptional quadric of the big resolution.
\end{lemma}
Granting Lemma \ref{lAlgebraicity} for the moment, we see then that all of the $2^s$ small resolutions are algebraic varieties: we can simply glue the local algebraic resolutions which exist around each of the nodes. Now to prove Lemma \ref{lAlgebraicity}, we distinguish the two cases:  

Suppose first that $Y$ contains a plane $E$, which contains the node $P$. In a small affine neighborhood $U_P$ of $P=0$ we can write the equation of $Y$ as (we use inhomogeneous coordinates $x, y, z, w$)
\[
y q_1 - x q_2 = 0
\]
where 
\begin{gather*}
q_1 = z + \mathrm{quadratic} \; \mathrm{terms} \; \mathrm{in}\; x, y, z, w, \\
q_2 = w + \mathrm{quadratic} \; \mathrm{terms} \; \mathrm{in}\; x, y, z, w .
\end{gather*}

Then we have the descriptions of the small resolutions locally above $P=0$ as

\begin{gather*}
\mathbb{A}^4 \times \PP^1 \supset U_{1, P}:=  \left\{   \left( (x, y, z, w), (\xi : \eta) \right)  \mid \left( \begin{array}{cc} q_1 & q_2 \\   x & y \end{array}\right)\left( \begin{array}{c} \xi \\ \eta \end{array} \right) = 0  \right\}
\end{gather*}

and 

\begin{gather*}
\mathbb{A}^4 \times \PP^1 \supset U_{2, P}:=  \left\{   \left( (x, y, z, w), (\xi : \eta) \right)  \mid  (\xi , \eta ) \left( \begin{array}{cc} q_1 & q_2 \\   x & y \end{array}\right) = 0  \right\} .
\end{gather*}

Now suppose we are in the case where $Y$ has six nodes, but does not contain a plane. Then by \cite{FinkWern89}, \S 3, analysis of case J9,  $Y$ has two projective (global) resolutions $Y_1$ and $Y_2$, and we can take for $U_{1,p}$ resp. $U_{2, P}$ the preimages of a small Zariski neighborhood of $P \in Y$ inside $Y_1$ resp. $Y_2$.  This concludes the proof of Lemma \ref{lAlgebraicity} and with this the proof of Lemma \ref{lAlgebraic}.
\end{proof}

We now consider first the case where $Y\subset \PP^4$ is a cubic with $6$ nodes and investigate the uniform rationality of its algebraic small resolutions; after that we turn to cubics where every node lies on some plane contained in the cubic.

\

We say that a finite set of points in $\PP^n$ is in general position if every subset of $m\le n+1$ of them spans a $\PP^{m-1}$. In the case of Lemma \ref{lFW}, the six points must be in general position. In fact, we recall from \cite{CLSS99} (Lemma 2.1)

\begin{lemma}
If a cubic threefold $Y \subset \PP^4$ has nodes $P_1, \dots , P_d$ which are not in general position, then $Y$ contains a plane $E$ containing at least three of the $P_i$.
\end{lemma} 

Now it follows from \cite{H-T10}, Proposition 19, that a cubic with $Y$ with $6$ nodes is determinantal (this was already known to C. Segre, see \cite{Segre87}), i.e. given by
\[
\det A(x_0, \dots , x_4) = 0
\]
where $A = A(x_0, \dots , x_4)$ is a  $3\times 3$ matrix of linear forms \[ A = (l_{ij}(x_0, \dots , x_4))_{1\le i,j\le 3} . \] 

\begin{proposition}\xlabel{pSixPoints}
Let $Y\subset \PP^4$ be a cubic with six nodes. Then an algebraic small resolution $X$ of $Y$ is uniformly rational.
\end{proposition}

\begin{proof}
By \cite{FinkWern89} there are (at most ) two projective small resolutions (possibly isomorphic). We can now describe them as follows:
\begin{gather*}
\PP^4 \times \PP^2 \supset X_1 := \left\{  \left( (x_0: \dots : x_4), (\xi : \eta : \zeta ) \right)  \mid A(x_0, \dots , x_4) (\xi, \eta , \zeta )^t =0  \right\}
\end{gather*}
and 
\begin{gather*}
\PP^4 \times \PP^2 \supset X_2 := \left\{  \left( (x_0: \dots : x_4), (\xi : \eta : \zeta ) \right)  \mid (\xi, \eta , \zeta ) A(x_0, \dots , x_4) = 0  \right\} .
\end{gather*}
The cases are analogous and we treat $X=X_1$. We claim that $X_1$ is a Zariski locally trivial projective bundle over $\PP^2$ of rank $1$ (via the projection to the second factor): indeed, the fiber over $(\xi : \eta : \zeta )$ is given by the vanishing of three linear forms in $\PP^4$, the set of zeros of the generalized column of $A$ that arises by taking the linear combination of the columns with coefficients $(\xi, \eta , \zeta )$. Suppose the three linear forms were not independent: then, by suitable invertible column and row operations, one can transform $A$ into a matrix which has a zero entry in some position. Then Laplace expansion shows that $Y$ contains a plane, contradicting the assumption. 

For $X_2$ the argument is the same, interchanging the role of columns and rows. 

An arbitrary algebraic resolution of $Y$ (Zariski-)locally looks like one of these two projective resolutions in this case. Therefore we get the assertion in general. 
\end{proof}

By Lemma \ref{lFW}, we may thus assume that $Y$ contains a plane $E$, which contains the node $P$. In an affine neighborhood $U_P$ of $P=0$ we have the descriptions of the small resolutions locally above $P=0$ as $U_{1,P}$, $U_{2,P}$ as at the end of the proof of Lemma \ref{lAlgebraic}.  We get projections 
\[
\pi_1 : U_{1, P} \to \PP^1 , \; \pi_2 : U_{2, P} \to \PP^1 . 
\]
Here $U_{1, P}$ compactifies to $\bar{\pi}_1 : \bar{U}_{1, P} \to \PP^1$, a quadric fibration over $\PP^1$, whereas $U_{2, P}$ compactifies to $\bar{\pi}_2 : \bar{U}_{2, P} \to \PP^1$ a fibration in degree $4$ del Pezzo surfaces (intersections of two quadrics in $\PP^4$) over $\PP^1$.

Note that in case of $U_{1, P}$ the strict transform of the plane $x=y=0$ contains the exceptional $\PP^1$'s whereas in the case of $U_{2, P}$, the strict transform of that plane intersects the exceptional $\PP^1$'s in a point (with coordinate $(0:1)$ in each case). Also notice that for any algebraic small resolution $X$ of $Y$, $P$ a node of $Y$ in $E$, $X$ will, locally above $P$, be isomorphic to $U_{1, P}$ or $U_{2, P}$.

An idea to prove uniform rationality in these cases is the following: in the first case we may try to project from a generic section of $\bar{\pi}_1$, in the second case we want to project from a section of $\bar{\pi}_2$ not meeting any lines in fibers to transform the fibration into a fibration in cubic surfaces over $\PP^1$; then we can choose another section in this cubic surface fibration and consider the relative reflection in this section. 

However, there is a simpler approach. Suppose (in homogeneous global coordinates)
\[
Y = \{ lA - mB = 0 \}
\]
where $l, m$ are linear forms, $A, B$ quadratic forms in $x_0, \dots , x_4$. Then, following \cite{Segre87}, we can consider $Y$ as a projection from the point $(0:0:0:0:0:1)$ of the variety $X$ in $\PP^5$ given by 
\[
x_5 l - B = 0, \; x_5 m - A =0 .
\]
On this birational model $X$ of $Y$, the four nodes in the plane $l=m=0$ are replaced by $\PP^1$'s whereas the plane itself is contracted to $(0:0:0:0:0:1)$. Moreover, $X$ is smooth in a neighborhood of the four exceptional $\PP^1$'s and a local model of the small resolution $U_{2, P}$ (the case where the strict transform of the plane intersects the exceptional $\PP^1$'s transversely) is obtained by blowing up $(0:0:0:0:0:1)$. Since the smooth locus of $X$ is uniformly rational by Example \ref{eQuadrics}, and the property of being uniformly rational is stable under blow-ups along smooth centers by \ref{pBlowup}, we get that the small resolution $U_{2, P}$ is uniformly rational.

For $U_{1, P}$ we may consider the variety
\[
\PP^4 \times \PP^1 \supset X : = \left\{   \left( (x_0: \dots : x_4), (\xi : \eta ) \right)  \mid \left( \begin{array}{cc} A & B \\ m & l \end{array}\right) \left( \begin{array}{c} \xi \\ \eta \end{array} \right) = 0   \right\}
\]
with projections to $Y \subset \PP^4$ and to $\PP^1$. Over a Zariski neighborhood of the four nodes in $Y$ in the chosen plane, $X$ is isomorphic to $U_{1, P}$. Hence, uniform rationality in this case follows from 

\begin{lemma}\xlabel{lQuadricFibration}
Let $\mathcal{X} \to \PP^1$ be a fibration in two-dimensional quadrics in a projective bundle $\PP (\mathcal{E}) \to \PP^1$ (where $\mathcal{E}$ is some rank $4$ vector bundle over $\PP^1$). Then the smooth locus of $\mathcal{X}$ is uniformly rational.
\end{lemma}

\begin{proof}
Suppose first that $x \in \mathcal{X}$ is a point in the total space in a neighborhood of which all fibers are smooth quadrics. Then we can see that such a point has a Zariski open neighborhood isomorphic to one in affine space by choosing a local section of the quadric fibration passing through the fiber in which $x$ lies, but not through the point $x$. Then we relatively project from that section.

If $x$ lies on a singular fiber, we have to study local normal forms of the possible degenerations of such quadric fibrations near $x$. Generally, for quadrics in $\PP^n$, they are given by
\[
x_0 + \dots + x_r  + t (x_{r+1} + \dots + x_s) , 
\]
for $t \to 0$ where $0 \le s \le n$ plus $1$ is the generic rank of the quadrics in the family. Let us analyze first what happens if the generic quadric in our family is smooth. Then we have two cases where we can write the degenerations respectively as
\[
x_0x_1 - t x_2x_3 = 0 
\]
(degeneration to two planes meeting transversally) or
\[
x_0x_1 + x_2^2 + tx_3^2 =0
\]
(degeneration to a quadric cone). The case of a degeneration to a double plane is not interesting because then the total space will be singular along the points of the fiber containing the double plane. 

\

In the quadric cone case, the family compactifies to 
\[
\mathcal{Q}_1 = \{ s(x_0x_1 + x_2^2) + tx_3^2 =0 \}
\]
where $(s:t)$ are now homogeneous coordinates on $\PP^1$. The total space of $\mathcal{Q}_1$ is smooth away from the fiber over $(s:t) = (0:1)$. The points of that fiber which are singular on the total space are precisely those lying on the conic which has $x_0x_1 + x_2^2=0$ and $x_3=0$.  Projection to $\PP^3$ shows that $\mathcal{Q}_1$ is obtained in the following way: blow up $\PP^3$ along the nonreduced subscheme given by $x_3^2=0$, $x_0x_1 + x_2^2=0$; this is the conic $C$ given by $x_0x_1 + x_2^2=0$ in the $x_3=0$ plane infinitesimally thickened in the $x_3$-direction. Geometrically, this means that one blows up $\PP^3$ first in the reduced conic to obtain a smooth variety $Z_1$. This contains the projectivisation of the normal bundle of the conic as an exceptional divisor $E_1$, and inside this exceptional divisor one has the locus of points corresponding to the directions in which $C$ is infinitesimally thickened, which is another conic $C_1$. Now blow up $X_1$ in $C_1$ to obtain a smooth variety $X_2$ (which is uniformly rational). The strict transform of $E_1$ on $X_2$ has normal bundle restricted to one of the $\PP^1$'s in which $E_1$ is fibered equal to $\mathcal{O}(-2)$. We contract this strict transform and obtain precisely $\mathcal{Q}_1$ which is singular along a conic, the contracted strict transform of $E_1$. 

\

In the case of the two planes meeting transversally, one can compactify the family to
\[
\mathcal{Q}_2 = \{ sx_0x_1 - t x_2x_3 =0 \}
\]
where again $(s:t)$ are now homogeneous coordinates on $\PP^1$. Projection to $\PP^3$ shows that in this case, the total space of the compactification is the blow up of $\PP^3$ in the quadrangle of lines given by $x_0x_1 = x_2x_3 =0$. Now this blow-up $\mathcal{Q}_2$ is uniformly rational by Proposition \ref{pBlowup} away from the locus of $\mathcal{Q}_2$ which lies over the four vertices of the quadrangle, located at where all but one of the $x_i$ vanish, and where the blow-up center in $\PP^3$ is not smooth. Moreover, $\mathcal{Q}_2$ is also uniformly rational away from the fibers over $(1:0)$ and $(0:1)$. But the eight points where $(s:t)= (0:1)$ or $(1:0)$ and all but one of the $x_i$ vanish is exactly the singular locus of the total space of $\mathcal{Q}_2$. Hence its entire smooth locus is uniformly rational.

It remains to consider the case where the general quadric in the fibration $\mathcal{X} \to \PP^1$ is not smooth. In this case, the family is a cone over the degeneration of a family of plane conics to two lines or a double line, which can be written as
\[
\mathcal{C} = \{ s x_0 x_1 + t x_2^2 = 0 \} .
\]
We claim that the smooth locus of the total space of this conic fibration is uniformly rational, which will also imply the result for the smooth locus of the cone over it. $\mathcal{C}$ is only singular in the two points above $(0:1)$ given by $x_0=x_2=0$ and $x_1=x_2=0$. Geometrically, $\mathcal{C}$ is the blow-up (via the projection to $\PP^2$) of $\PP^2$ in the nonreduced subscheme $x_0=x_2^2=0$ and $x_1=x_2^2=0$ which are two points on the line $x_2=0$ with an infinitesimal thickening in the $x_2$ direction. Geometrically, we blow up $\PP^2$ in the corresponding reduced points, then blow up the resulting space again in the points of the exceptional $\PP^1$'s, $E_1$, $E_2$, corresponding to the direction of the infinitesimal thickenings, and finally contract the strict transforms of $E_1$, $E_2$ on the second blow-up where they have become $(-2)$-curves. This gives $\mathcal{C}$ and proves that its smooth locus is uniformly rational, using Proposition \ref{pBlowup}.
\end{proof}

Hence in summary, we obtain 

\begin{theorem}\xlabel{tSmallCubic}
All small algebraic resolutions of nodal cubic threefolds $Y \subset \PP^4$ are uniformly rational.
\end{theorem}

\subsection{Alternative proof that small resolutions are uniformly rational}\xlabel{ssAlternative}
There is an alternative and shorter proof that all small algebraic resolutions of nodal cubic threefolds $Y \subset \PP^4$ are uniformly rational; however, it uses heavily the geometric picture developed in the preceding subsection. Moreover, the determinantal constructions in the preceding subsections are more explicit and we obtained results of independent interest such as Lemma \ref{lQuadricFibration} in the course of the argument there. Therefore we think it worthwhile to include both methods.

The crucial observation is simply that if a small resolution is algebraic, all nodes on $Y$ correspond to double points $q$ of the associated curve $C$ where \emph{two irreducible smooth components} $C_1$ and $C_2$ of $C$ meet transversely. If we choose a Zariski open $U$ around $q$ in $\PP^3$ in which $C_1$ and $C_2$ are smooth, then we can first blow up $U$ in $C_1$ and then in $C_2$, or in the other order, first in $C_2$, then in $C_1$. This gives us two uniformly rational varieties, and each of them is isomorphic to an open neighborhood of the exceptional $\PP^1$ over the node of $Y$ corresponding to $q$ in one or the other small resolution that one gets locally above that node. The result follows.

\subsection{Nonalgebraic Moishezon resolutions}\xlabel{ssMoishezon}

Here we show that  Question \ref{qGromov2} has in general a negative answer for the nonalgebraic small resolutions of nodal cubic threefolds. 

\begin{proposition}\xlabel{pMoishezon}
Let $Y$ be a Lefschetz cubic with one singular point $P$, $X$ a small resolution. Then $X$ is a non-algebraic Moishezon manifold where for no point $x$ of the exceptional curve one can find a bimeromorphic map $X \dasharrow \PP^3$ defined at $P$. 
\end{proposition}

\begin{proof}
As remarked in \cite{Moi67}, the exceptional $C=\PP^1$ is homologous to zero in this case. Hence every compact complex surface $S$  contained in $X$ and intersecting $C$ must contain $C$ entirely. This excludes the existence of algebraic charts around points of $C$ (otherwise we could take for $S$ the birational transform of a surface intersecting the image of $C$ in that chart in a finite number of points) and it also excludes the existence of bimeromorphic maps $X \dasharrow \PP^3$ defined at a point of $C$: otherwise we could take for $S$ the bimeromorphic transform of some surface in $\PP^3$ meeting the closure of the image of $C$ in $\PP^3$ in a finite number of points.
\end{proof}

This argument also shows that all small non-algebraic resolutions of nodal cubic threefolds do not satisfy what is asked in Question \ref{qGromov2} because always one node is resolved zero-homologously by Remark \ref{rAlgebraicity}, (2), and Lemma \ref{lAlgebraic}.

\subsection{Big resolutions of nodal cubic threefolds}\xlabel{ssBigResolutions}

Here we prove

\begin{theorem}\xlabel{tBigResolutions}
Let $X \subset \PP^4$ be a nodal cubic threefold, $p\in X$ a node, and $\hat{X}=\mathrm{Bl}_p (X)$ the big resolution of $p$. Then $\hat{X}$ is uniformly rational in a neighborhood of the exceptional divisor $Q \simeq \PP^1 \times \PP^1$. 
\end{theorem}

\begin{proof}
If $X$ has only $p$ as a single node, then this follows from Proposition \ref{pBigRes}. Moreover, it suffices to prove that every point $x\in Q$ has a Zariski open neighborhood isomorphic to a Zariski open subset in $\PP^3$ by the discussion in Example \ref{eCubics}. 

Consider the blow up of $\PP^4$ in $p$ and $\hat{X} \subset \mathrm{Bl}_p (\PP^4)$. Choose a general line $l \subset X$ passing through $p$. This corresponds to choosing a general point of the associated curve $C$. We want to look at the projection $\pi_l : X \dasharrow \PP^2$ from $l$ more closely. Let $\hat{l} \subset \hat{X}$ be the strict transform of $l$. Then $\pi_l$ induces a morphism
\begin{gather*}
p : \tilde{X} = \mathrm{Bl}_{\hat{l}} (\hat{X}) \to \mathbb{P}^2 .
\end{gather*}
We can assume that this is a conic bundle, in particular, flat: namely we may assume without loss of generality that there is no plane containing $l$ and entirely contained in $X$. If this were the case, then, locally above $p$, both small resolutions of $p$ would be algebraic, and the big resolution is a blow up of both of these. But we know already that small resolutions are uniformly rational and that the property is stable under blow ups in smooth centers.

Moreover, notice that the strict transform $\tilde{Q}$ of $Q$ on $\tilde{X}$ is the quadric $Q\simeq \PP^1 \times \PP^1$ blown up in one point (namely the point $P_l$ corresponding to the direction given by $l$), hence isomorphic to the blow up of $\PP^2$ in two points $a, b$. The exceptional divisors above $a$ and $b$ are simply the strict transforms of the two lines $L_1$, $L_2$ of the two rulings on $Q$ passing through the point $P_l$ corresponding to $l$.  The map
\[
p|_{\tilde{Q}} : \tilde{Q} = \mathrm{Bl}_{a, b} (\PP^2 ) \to \PP^2
\]
is simply the blow down map. In particular, the strict transforms of $L_1$ and $L_2$ map to points since the lines in the tangent cone of $X$ at $p$ corresponding to points in $L_1$ or $L_2$ all lie in the same plane with $l$.

Note now that outside of the exceptional divisor over $P_L$, $\tilde{Q}$ is isomorphic to $Q$, and an open neighborhood of $\tilde{Q}$ minus the total transforms of $L_1$, $L_2$ in $\tilde{X}$ is of course isomorphic to an open neighborhood of $Q\backslash (L_1\cup L_2)$ in $\hat{X}$. Moreover, letting $l$ vary over all points of the associated curve $C$, the open subsets $Q\backslash (L_1\cup L_2)$ cover all of $Q$: this could only fail if the associated curve consisted of a single component which is a line of one ruling of $Q$, which never happens since it is of bidegree $(3, 3)$. 

Hence it suffices now to prove: 
\begin{quote}
\textbf{Claim:} 
Let $\tilde{L}_i$ be the strict transform of $L_i$ on $\tilde{Q}$. Then every point of $\tilde{Q}^0:= \tilde{Q} \backslash (\tilde{L}_1 \cup \tilde{L}_2)$ has a Zariski open neighborhood in $\tilde{X}$ which is isomorphic to a Zariski open in $\PP^3$. 
\end{quote}

Suppose $\tilde{L}_i$ maps to $r_i\in \PP^2$. Then $\tilde{Q}^0$ is a section of 
\[
\tilde{X}^0 := \tilde{X} \backslash (p^{-1}(r_1) \cup p^{-1}(r_2))  \to (\PP^2 )^0:= \PP^2\backslash (r_1 \cup r_2). 
\]

Now remark that $\tilde{Q}^0$ will intersect each fiber of $p$ in a smooth point of that fiber: namely,  the existence of the section implies a splitting of the tangent space $T_x \tilde{X} \simeq T_{p(x)} S \oplus N_x$ and shows that $T_{y} \tilde{X} \to T_{p(y)} S$ is a submersion in a neighborhood of $y$. Hence the fiber of $p$ through $y$ is smooth at $y$. Compare also \cite{A-B-B11}, Lemma 2.8 for a generalization of this observation.

 To prove the claim, we now use the group scheme $\mathcal{G}$ over $S=(\PP^2)^0$ whose fiber $\mathcal{G}_s$ over a point $s\in S$ is the connected component of the identity of the automorphism group of the fiber $\tilde{X}_s$, which is a conic. If this  conic is smooth we get just $\mathrm{SO}_3 (\CC )$, whereas, if the fiber is a union of two lines meeting transversely, we get the automorphism group of the quadratic form $xy$ (where $x,y,z$ are coordinates in $\PP^2$). This is an extension by the additive group $\CC^2$ of the group $\ZZ/2 \ltimes (\CC)^*$. One directly checks that its connected component of the identity acts transitively on each of the two lines of the degenerate conic. The same is of course true for the action of $\mathrm{SO}_3 (\CC )$ on a smooth conic. 
 
Now we know already that each point in $\tilde{X}^0$ outside of $\tilde{Q}^0$ has a Zariski open neighborhood isomorphic to an open set in $\PP^3$. 
Also we know that every point of $\tilde{Q}^0$ on a smooth fiber has such a Zariski open neighborhood: the conic bundle is Zariski locally trivial there since it has a smooth section, hence we can use a relative projection from some generic section of it defined locally around that point to see this. 

Thus we can concentrate on points of $\tilde{Q}^0$ which lie on a fiber $p^{-1}(P)$, $P\in\PP^2$, which consists of two lines meeting transversely. We now look at $\mathcal{G}$ over the local ring $\mathcal{O}_{P}$ of the base. We have to show that every element of the fiber $\mathcal{G}_P$ lifts to an element of $\mathcal{G} \times_{S^0} \mathrm{Spec} (\mathcal{O}_{P})$. Since the conic bundle has a smooth section, it suffices to check that we can lift over the completion $\hat{\mathcal{O}}_{P}$, hence we can assume the degeneration is locally given by $xy + t z^2 = 0$ where $t=0$ is the degeneration divisor. One can then directly compute the orthogonal group over  $\hat{\mathcal{O}}_{P}$ that stabilizes this form and check that each element of the group scheme fiber over $t=0$ lifts. 

Hence, locally around the fiber over $P$,  there exists a fiberwise automorphism of $\tilde{X}^0$ that moves the point of $\tilde{Q}^0$ on the degenerate fiber into a Zariski open around some point off $\tilde{Q}^0$ which is isomorphic to a Zariski open subset of $\PP^3$. Hence the claim follows and Theorem \ref{tBigResolutions} follows. 
\end{proof}



\begin{thebibliography}{99999999}

\bibitem[A-B-B11]{A-B-B11}
A. \ Auel, M. \ Bernardara, M. \ Bolognesi, 
\emph{Fibrations in complete intersections of quadrics, Clifford algebras, derived categories, and rationality problems}, preprint (2011), available at arXiv:1109.6938

\bibitem[CaVa97]{CaVa97}
L. \ O'Carroll \& G. \ Valla, \emph{On the smoothness of blow ups}, Communications in Algebra \textbf{25} (6), (1997) 1861-- 1872

\bibitem[CLSS99]{CLSS99}
D.F. \ Coray,  D.J. \ Lewis, N.I. \  Shepherd-Barron, P.\  Swinnerton-Dyer, 
\emph{Cubic threefolds with six double points}, Number theory in progress, Vol. \textbf{1} (Zakopane-Koscielisko, 1997), de Gruyter, Berlin, (1999), 63--74

\bibitem[Fink87]{Fink87} H. \ Finkelnberg, 
\emph{Small resolutions of the Segre cubic}, 
Nederl. Akad. Wetensch. Indag. Math. \textbf{49} (1987), no. 3, 261--277

\bibitem[FinkWern89]{FinkWern89}
H. \ Finkelnberg \& J. \ Werner,
\emph{Small resolutions of nodal cubic threefolds}, Nederl. Akad. Wetensch. Indag. Math. \textbf{51} (1989), no. 2, 185--198

\bibitem[FuNa71]{FuNa71} A. \ Fujiki \& S. \ Nakano, 
\emph{ Supplement to ``On the Inverse of Monoidal Transformation"}
,  Publ. RIMS Kyoto Univ. Vol. \textbf{7}, (1971/72), 637--644

\bibitem[Gro89]{Gro89} M.\ Gromov,
\emph{Oka's Principle for Holomorphic Sections of Elliptic Bundles}, Journal of the American Mathematical Society, Vol. \textbf{2}, No. 4 (1989), 851--897

\bibitem[H-T10]{H-T10}
B. \ Hassett, Y. \ Tschinkel, 
\emph{Flops on holomorphic symplectic fourfolds and determinantal cubic hypersurfaces}, 
J. Inst. Math. Jussieu \textbf{9} (2010), no. 1, 125--153

\bibitem[Hunt96]{Hunt96}
B. \ Hunt, \emph{The geometry of some special arithmetic quotients}, LNM \textbf{1637}, Springer- Verlag (1996)

\bibitem[Keel92]{Keel92} S. \ Keel, 
\emph{Intersection Theory of Moduli Space of Stable N-Pointed Curves of Genus Zero}, Transactions of the American Mathematical Society, Vol. \textbf{330}, No. 2 (1992), 545--574

\bibitem[MelPol09]{MelPol09} M. \ Mella \& E. \ Polastri, 
\emph{Equivalent birational embeddings}, Bull. Lond. Math. Soc. \textbf{41} (2009), no. 1, 89--93

\bibitem[MelPol12]{MelPol12} M. \ Mella \& E. \ Polastri, 
\emph{Equivalent birational embeddings II: divisors}, Math. Zeitschrift  \textbf{270} (2012), 1141--1161

\bibitem[Mel12]{Mel12} M. \ Mella, 
\emph{Equivalent birational embeddings III: cones}, preprint October 2012, 8 pages

\bibitem[Moi67]{Moi67} B.G. Moishezon,
\emph{On $n$-dimensional compact varieties with $n$ algebraically independent meromorphic functions, I, II, III}, American mathematical Society Translations, Series 2, Volume \textbf{63}, (1967), 51--177


\bibitem[Nak70]{Nak70} S. \ Nakano, 
\emph{On the Inverse of Monoidal Transformation}, Publ. RIMS Kyoto Univ. Vol. \textbf{6} (1970/71), 483--502

\bibitem[Segre87]{Segre87}
C.\ Segre, \emph{Sulle variet\`{a} cubiche dello spazio a quattro dimensioni e su certi sistemi di rette e certe superficie dello spazio ordinario}, Memorie della Reale Accademia delle Scienze di Torino, serie II, tomo XXXIX, (1887), 3--48. Reprinted in ``Opere. Volume IV", Edizioni Cremonese, Roma (1963)

\bibitem[Wern87]{Wern87}
J.\  Werner, \emph{Kleine Aufl\"osungen spezieller dreidimensionaler Variet\"aten}, Thesis, Bonner Mathematische Schriften \textbf{186},  Universit\"at Bonn, Mathematisches Institut, Bonn, (1987), viii+119 pp.
\end{thebibliography}
\end{document}